\documentclass{article}
\usepackage[utf8]{inputenc}
\usepackage{amsmath}
\usepackage{amsfonts} 
\usepackage{amsthm}
\usepackage{verbatim}
\usepackage{etoolbox}
\usepackage{dirtytalk}
\usepackage{tikz-cd}
\usepackage{derivative}
\usepackage{amssymb}
\usepackage[backend=biber,
style=numeric,giveninits=true]{biblatex} 
\theoremstyle{plain}
\newtheorem{thm}{Theorem}[section]
\newtheorem{lem}[thm]{Lemma}
\newtheorem{obs}[thm]{Observation}

\newtheorem{cor}[thm]{Corollary}

\newtheorem*{claim*}{Claim}

\theoremstyle{definition}
\newtheorem{defi}[thm]{Definition}
\newtheorem{ex}[thm]{Example}

\theoremstyle{remark}
\newtheorem{rem}{Remark}[thm]

\AtBeginEnvironment{thm}{\setcounter{case}{0}}
\AtBeginEnvironment{lem}{\setcounter{case}{0}}
\AtBeginEnvironment{obs}{\setcounter{case}{0}}
\AtBeginEnvironment{defi}{\setcounter{case}{0}}
\AtBeginEnvironment{prop}{\setcounter{case}{0}}
\AtBeginEnvironment{cor}{\setcounter{case}{0}}

\AtBeginEnvironment{thm}{\setcounter{claim}{0}}
\AtBeginEnvironment{lem}{\setcounter{claim}{0}}
\AtBeginEnvironment{obs}{\setcounter{claim}{0}}
\AtBeginEnvironment{defi}{\setcounter{claim}{0}}
\AtBeginEnvironment{prop}{\setcounter{claim}{0}}
\AtBeginEnvironment{cor}{\setcounter{claim}{0}}

\newcommand{\R}{\mathbb{R}}

\newcommand{\RR}{\mathcal{R}}
\newcommand{\s}{\mathbb{S}}

\newcommand{\id}{\mathrm{id}}

\newcommand{\alg}{\mathrm{alg}}
\DeclareMathOperator{\sym}{sym}

\title{Homotopy properties of regular mappings into real retract rational varieties}
\author{Juliusz Banecki}
\date{}
\AtEndDocument{%
  \par
  \textsc{Juliusz Banecki,
Faculty of Mathematics and Computer Science,
Jagiellonian University,
ul. Lojasiewicza 6, 30-348 Krakow, Poland}\\
\textit{E-mail address}: 
\texttt{juliusz.banecki@student.uj.edu.pl}
}

\begin{document}
\maketitle
\renewcommand{\thefootnote}{}%
\footnotetext{2020 Mathematics Subject Classification. 14P25, 55Q05.}
\footnotetext{Keywords: regular mapping, retract rational variety, homotopy group.}
\renewcommand{\thefootnote}{\arabic{footnote}}%
\begin{abstract}
We study homotopy properties of regular mappings from spheres into a real retract rational variety $Y$. We show that the homotopy classes which are represented by such mappings form subgroups of the homotopy groups of $Y$, and that the groups are independent of the choice of the basepoint on $Y$ as long as $Y$ is connected. We also construct regular representatives of all the Whitehead products in all the homotopy groups of $Y$.
\end{abstract}
\section{Introduction}
Throughout this paper, a \emph{real affine variety} is a topological space equipped with a sheaf of functions with values in $\R$, isomorphic to an algebraic subset of some $\mathbb{R}^n$ endowed with the Zariski topology and the sheaf of regular functions. Morphisms between real affine varieties are called \emph{regular mappings}. Note that, unlike in complex algebraic geometry, every real quasi-projective variety is isomorphic to an affine one. The ring of globally defined regular functions on an affine variety $X$ is denoted by $\RR(X)$. A point $x\in X$ is said to be nonsingular if the corresponding local ring of germs of regular functions $\RR(X)_x$ is a regular local ring. These definitions are compatible with the ones given in \cite{bochnakRealAlgebraicGeometry1998,mangolteRealAlgebraicVarieties2020}.

Our main object of study is the following class of \emph{retract rational varieties}, first introduced by Saltman in \cite{saltmanMultiplicativeFieldInvariants1987}:
\begin{defi}
An affine variety $Y$ is said to be retract rational if there exists a Zariski open and dense subset $U\subset Y$, a Zariski open subset $V\subset \R^n$ for some $n$ and two regular mappings $i:U\rightarrow V, r:V\rightarrow U$ such that their composition
\begin{equation*}
    U\xrightarrow{i} V\xrightarrow{r} U 
\end{equation*}
equals the identity on $U$.
\end{defi}
It is a simple consequence of the definition, that this class of varieties contains all the \emph{rational} and more generally all the \emph{stably rational} ones (an affine variety $Y$ is called stably rational if $Y\times \R^n$ is rational for some $n\geq 0$).

It has recently been realised, through the results of \cite{baneckiRelativeStoneWeierstrassTheorem2024,baneckiRetractRationalVarieties2025}, that the class of retract rational varieties is of great use when studying approximation properties of regular mappings between real affine varieties. The current paper aims to extend these results even further, using tools developed in the aforementioned papers and bringing new ideas into the field. 

We begin by summarising the recent results regarding real retract rational varieties, which will be of some use in the current paper and are interesting on their own. The starting point of the summary is the following theorem:
\begin{thm}[{\cite[Theorem 1.8]{baneckiRetractRationalVarieties2025}}]\label{baneckiRetractRationalVarieties2025}
Let $Y$ be a nonsingular retract rational affine variety. Then it is \emph{uniformly retract rational}, meaning that for each point $y_0\in Y$ there exists a Zariski open neighbourhood $U$ of $y_0$ in $Y$, a Zariski open subset $V\subset \R^n$ for some $n$ and two regular mappings $i:U\rightarrow V, r:V\rightarrow U$ such that their composition
\begin{equation*}
    U\xrightarrow{i} V\xrightarrow{r} U 
\end{equation*}
is equal to the identity on $U$.
\end{thm}

Theorem \ref{baneckiRetractRationalVarieties2025} allows us to reformulate older results of \cite{baneckiRelativeStoneWeierstrassTheorem2024}, which regard uniformly retract rational varieties, so that now we can talk about nonsingular retract rational varieties instead. Without getting into too much detail, the main result of \cite{baneckiRelativeStoneWeierstrassTheorem2024} can be formulated as follows:
\begin{thm}\label{approx_thm}
Let $Y$ be a nonsingular retract rational affine variety, let $X$ be an affine variety, let $Z\subset X$ be a Zariski closed subvariety of $X$ compact in the Euclidean topology, and let $f:X\rightarrow Y$ be a continuous mapping satisfying the following two conditions:
\begin{enumerate}
    \item $f$ is homotopic to a regular mapping $g:X\rightarrow Y$,
    \item $f\vert_Z:Z\rightarrow Y$ is regular. 
\end{enumerate}
Then $f$ can be approximated by regular mappings $\widetilde{f}:X\rightarrow Y$ in the compact-open topology, agreeing with $f$ on $Z$.
\end{thm}

The result is an important tool used in \cite{kucharzSpacesMapsReal2025}, where Kucharz was able to characterise the weak homotopy type of the connected components of the space $\RR(X,Y)$ of regular mappings from an affine variety $X$ to a nonsingular retract rational variety $Y$ equipped with the compact-open topology. It also plays a major role in \cite{bilskiApproximationMapsAlgebraic2025}, where approximation of continuous mapping by piecewise regular ones is considered.

Given these results regarding approximation properties of regular mappings, in this paper we come to study their homotopy properties. Our main results deal with \emph{algebraic homotopy classes}, studied before in \cite{bochnakAlgebraicApproximationMappings1987,bochnakRealizationHomotopyClasses1987,pengAlgebraicMapsSpheres1999} and most recently in \cite{baneckiAlgebraicHomotopyClasses2024,golasinskiSpheresFieldsTheir2024}. To introduce them, we need to fix some notation. Let $\s^n$ denote the $n$-dimensional unit sphere in $\R^{n+1}$, and let $e:=(1,0,\dots,0)\in\s^n$ be the north pole of $\s^n$. Clearly, $\s^n$ is an algebraic subset of $\R^{n+1}$, which allows us to naturally endow it with the structure of an affine variety. Given another affine variety $Y$ with a fixed point $y_0\in Y$, we denote by $\pi_n(Y,y_0)$ the set of homotopy classes of continuous mappings of pointed topological spaces $f:(\s^n,e)\rightarrow (Y,y_0)$. For $n\geq 1$, the set admits a natural group structure and is called the $n$-th homotopy group of $(Y,y_0)$. By $\pi_n^\alg(Y,y_0)$ we denote the set of \emph{algebraic homotopy classes}, i.e. those elements of $\pi_n(Y,y_0)$ which admit representatives $\tilde{f}:(\s^n,e)\rightarrow (Y,y_0)$ being regular mappings. 

Our main result can now be stated as follows:
\begin{thm}\label{thm_groups}
Let $Y$ be a nonsingular retract rational affine variety with a distinguished point $y_0\in Y$. Then $\pi_n^\alg(Y,y_0)$ is a subgroup of the homotopy group $\pi_n(Y,y_0)$ for every $n\geq 1$.
\end{thm}
\begin{rem}
As we show in Example \ref{ex_1}, the assumption about $Y$ being retract rational is essential.
\end{rem}

Such a result was earlier established in \cite[Theorem 1.1]{baneckiAlgebraicHomotopyClasses2024} in the special case when $Y$ is isomorphic to $\s^k$ for some $k$. Since spheres are nonsingular and retract rational, the current result is a significant generalisation of that (nonsingularity follows from the Jacobian criterion; for a short proof of retract rationality note that the stereographic projection gives an isomorphism between a Zariski open and dense subset of the sphere and the affine space). A version of the theorem is also known to hold for \emph{algebraic cohomotopy classes} which we will not be defining here; see \cite{baneckiAlgebraicHomotopyClasses2024} for the definition and some other results regarding it.

It is worth noting here, that Theorem \ref{approx_thm} implies that the groups $\pi_n(Y,y_0)$ are independent of the choice of the basepoint $y_0$ as long as $Y$ is connected:
\begin{obs}
Let $Y$ be a nonsingular retract rational variety which is connected (note that according to \cite[Observation 6.1]{baneckiRelativeStoneWeierstrassTheorem2024} this is the case whenever $Y$ is compact). Then, for $n\geq 1$, for any two points $y_0,y_1\in Y$ the groups $\pi^\alg_n(Y,y_0)$ and $\pi^\alg_n(Y,y_1)$ are isomorphic.
\begin{proof}
Let $\varphi:[0,1]\rightarrow Y$ be a path connecting $y_0$ and $y_1$. It induces an isomorphism $\varphi_\ast:\pi_n(Y,y_0)\xrightarrow{\cong}\pi_n(Y,y_1)$. Let $\alpha\in \pi_n^\alg(Y,y_0)$ be represented by a regular mapping $f:(\s^n,e)\rightarrow (Y,y_0)$ and let $g:(\s^n,e)\rightarrow (Y,y_1)$ be a continuous mapping representing $\varphi_\ast(\alpha)$. By construction $f$ and $g$ are homotopic. Hence, according to Theorem \ref{approx_thm}, the mapping $g$ can be approximated by a regular one $\tilde{g}:(\s^n,e)\rightarrow (Y,y_1)$. If the approximation is sufficiently close, it represents the same homotopy class as $g$. This shows that $\varphi_\ast\big(\pi_n^\alg(Y,y_0)\big)\subset\pi_n^\alg(Y,y_1)$. The other inclusion follows by an analogous argument.
\end{proof}
\end{obs}

Our next result, analogous to \cite[Observation 3.7]{baneckiAlgebraicHomotopyClasses2024}, deals with \emph{Whitehead products} (see e.g. \cite[Chapter X, Section 7]{whiteheadElementsHomotopyTheory1978} for the definition).
\begin{thm}\label{thm_products}
Let $Y$ be a nonsingular retract rational affine variety with a distinguished point $y_0\in Y$. Let $n,m\geq 1$ be natural numbers. Then, for any two classes $\alpha\in \pi_n(Y,y_0),\;\beta\in \pi_m(Y,y_0)$, their Whitehead product $[\alpha,\beta]\in\pi_{n+m-1}(Y,y_0)$ belongs to $\pi_{n+m-1}^\alg(Y,y_0)$.
\end{thm}
Representation of Whitehead products by regular mappings was considered before in \cite{pengAlgebraicMapsSpheres1999}, where it was shown to be possible for some specific homotopy classes of spheres. Later, in \cite{baneckiAlgebraicHomotopyClasses2024} it was shown to be possible for all the products in homotopy groups of spheres. The current result is an even further generalisation.

Our last result deals with mappings from nonsingular curves into nonsingular retract rational varieties:
\begin{thm}\label{thm_curves}
Let $X$ be a nonsingular affine curve, compact in the Euclidean topology, and let $Y$ be a nonsingular retract rational affine variety. Let $f:X\rightarrow Y$ be continuous. Then $f$ can be approximated by regular mappings $\tilde{f}:X\rightarrow Y$.
\end{thm}

Such a theorem was earlier known to hold in a few other cases: when $Y$ is stably rational (in particular, when $Y$ is rational), when $Y$ is birational to a smooth cubic hypersurface in $\R P^n$ for $n\geq 3$, and when $Y$ is birational to a homogenous space of a real algebraic group (see \cite[Theorem A and Corollary 4.4]{benoistTightApproximationProperty2021}). In particular, this does include examples of varieties which are not retract rational, e.g. the cubic surface in \cite[Example 6.3]{baneckiRelativeStoneWeierstrassTheorem2024}. It is expected that the theorem holds in general for all the \emph{rationally connected} varieties, 
see \cite[Question 4]{kucharzOpenQuestionsReal2022}. If true, this statement would imply all the results mentioned above. Here a variety $Y$ is said to be rationally connected if every two points of $Y$ lying in the same connected component of $Y$ can be connected by a rational curve in $Y$.

As a consequence we obtain a statement regarding algebraic classes in the fundamental group of $Y$:
\begin{cor}\label{cor:first_group}
Let $Y$ be a nonsingular retract rational affine variety and let $y_0\in Y$. Then $\pi_1^\alg(Y,y_0)=\pi_1(Y,y_0)$.
\begin{proof}
Let $f:(\s^1,e)\rightarrow (Y,y_0)$ be a representative of a given homotopy class $\alpha \in \pi_1(Y,y_0)$. According to Theorem \ref{thm_curves}, $f$ can be approximated by a regular mapping $\tilde{f}:\s^1\rightarrow Y$. If the approximation is sufficiently close, then $f$ is homotopic to $\tilde{f}$. Hence we can apply Theorem \ref{approx_thm} to find a regular mapping $\hat{f}:\s^1\rightarrow Y$ which agrees with $f$ at the point $e$ and approximates $f$. If the approximation is sufficiently close then $\hat{f}$ represents $\alpha$.
\end{proof}
\end{cor}

Given the positive results presented here, one could suspect that the equality $\pi_n^\alg(Y,y_0)=\pi_n(Y,y_0)$ holds true for all $n$ as long as $Y$ is nonsingular retract rational. At the end of the paper we provide Example \ref{ex_2}, which shows that this can fail to be true, even for $n=2$. This demonstrates how delicate our results are, especially when it comes to Theorem \ref{thm_groups}. The author sees the results as evidence suggesting that nonsingular retract rational varieties may comprise just the right class of varieties to study homotopy properties of regular mappings in real algebraic geometry.

\section{Preliminaries}

Recall the following technical lemma playing a role in \cite{baneckiRelativeStoneWeierstrassTheorem2024}:
\begin{lem}[{\cite[Lemma 3.2]{baneckiRelativeStoneWeierstrassTheorem2024}}]\label{gluing_lem}
Let $Y$ be an irreducible variety and let $n$ be a natural number. Let $U\subset Y\times\R^n$ be a Zariski open set such that $U\cap (Y\times\{0\})\neq\emptyset$. Let $f\in\RR(U)$ and let $P,Q\in\R(Y\times\R^n)$ be two rational functions, regular on a neighbourhood of $Y\times\{0\}$ satisfying
\begin{equation*}
    Qf=P\textnormal{ on a neighbourhood of } U\cap (Y\times\{0\}).
\end{equation*}
Suppose further that $f\vert_{U\cap (Y\times\{0\})}$ admits a regular extension onto the entire variety $Y\times\{0\}$. Then, the function
\begin{equation*}
    g(y,v):=f(y,Q(y,0)v)
\end{equation*}
defined on a neighbourhood of $U\cap (Y\times\{0\})$ admits a regular extension to a neighbourhood of $Y\times\{0\}$.
\end{lem}

We will need the following corollary of the result:

\begin{cor}\label{gluing_cor}
Let $Y$ be an irreducible variety, let $y_0\in Y$ and let $n$ be a natural number. Let $Z$ be another affine variety and let $F:(Y\times \R^n,(y_0,0))\rightarrow Z$ be a germ of a regular mapping at the point $(y_0,0)$. Then, there exists $q\in \RR(Y)$ such that $q(y_0)=1$ and the germ
\begin{gather*}
    G:(Y\times \R^n,(y_0,0))\rightarrow Z,\\
    G(x,v):= F(y,q(y)v),
\end{gather*}
extends as a regular mapping to a neighbourhood of $Y\times \{0\}$.
\begin{proof}
Embed $Z$ as a Zariski closed subset of the affine space $\R^m$ for some $m$. Since $F$ is regular in a neighbourhood of the point $(y_0,0)$, there exists a common denominator of all the coordinates of $F$, i.e. a regular function $Q\in \RR(Y\times \R^n)$, which can be assumed to attain value $1$ at $(y_0,0)$ such that 
\begin{equation*}
    QF=P\text{ in a neighbourhood of }(y_0,0),
\end{equation*}
for some regular mapping $P:Y\times\R^n\rightarrow \R^m$. It suffices to take $q(y):=Q(y,0)$ and apply Lemma \ref{gluing_lem}.
\end{proof}
\end{cor}

We also need the following simple observation:
\begin{obs}\label{embedding_lem}
Let $Y\subset \R^n$ be a nonsingular retract rational variety, embedded as a Zariski closed subset in $\R^n$. Fix a point $y_0\in Y$. Then, there exists a Zariski open set $W\subset \R^n$ containing $y_0$, and a regular mapping $R:W\rightarrow Y$ such that $R(y)=y$ for $y\in Y\cap W$.
\begin{proof}
According to Theorem \ref{baneckiRetractRationalVarieties2025}, there exist a neighbourhood $U$ of $y_0$ in $Y$, a Zariski open subset $V\subset \R^m$ for some $m$, and two regular mappings $i:U\rightarrow V,r:V\rightarrow U$ satisfying $r\circ i=\id_U$. The mapping $i$ admits a regular extension $\widetilde{i}$ to a Zariski open set $W\subset \R^n$ containing $U$ as a subset. It suffices to define $R:=r\circ \widetilde{i}$.
\end{proof}
\end{obs}

\section{Rational H-spaces}
To prove Theorem \ref{thm_groups} we introduce the following notion:

\begin{defi}
Let $Y$ be an irreducible affine variety and let $y_0\in Y$. We say that the pair $(Y,y_0)$ is a \emph{rational H-space} if there exist a Zariski open set $A\subset Y\times Y$ containing $Y\times\{y_0\}\cup\{y_0\}\times Y$ and a regular mapping $\mu:A\rightarrow Y$ satisfying $\mu(y_0,y_0)=y_0$ and such that
\begin{enumerate}
    \item $Y\ni y\mapsto\mu(y,y_0)\in Y$ is homotopic to the identity relative to $\{y_0\}$, 
    \item $Y\ni y\mapsto\mu(y_0,y)\in Y$ is homotopic to the identity relative to $\{y_0\}$.
\end{enumerate}
\end{defi}

\begin{thm}
Let $Y$ be a nonsingular retract rational variety, and let $y_0\in Y$ be a distinguished point. Then $(Y,y_0)$ is a rational H-space.
\begin{proof}
Fix an embedding of $Y\subset \R^n$ as a Zariski closed subset in $\R^n$ for some $n$. Let $W$ be the Zariski open subset of $\R^n$ containing $y_0$, and let $R:W\rightarrow Y$ be the regular mapping from Observation \ref{embedding_lem}. After a shift of coordinates we might assume that $y_0$ is the origin in $\R^n$.

Apply Corollary \ref{gluing_cor} to the following germ of a regular mapping at $(0,0)$:
\begin{align*}
    F:(Y\times \R^n,(0,0))\rightarrow Y, \\
    F(x,v):=R(x+v).
\end{align*}
This gives a Zariski open neighbourhood $U$ of $Y\times\{0\}$ in $Y\times\R^n$ and a regular function $q\in\RR(Y)$, such that $q(0)=1$ and the rational mapping
\begin{gather*}
    G:Y\times \R^n\dashrightarrow Y,\\
    G(x,v):= R(x+q(x)v),
\end{gather*}
extends as a regular one onto $U$. 

Using Łojasiewicz's inequality \cite[Proposition 2.6.2]{bochnakRealAlgebraicGeometry1998} we can find $N>0$ such that
\begin{equation*}
    \{(x,v)\in Y\times \R^m:  ||v||<||x||(1+||x||^2)^{-N}|q(x)|\}\subset U.
\end{equation*}
Define $s\in\RR(Y)$ by $s(x):=q^2(x)(1+||x||^2)^{-N}$, and note that $s$ is positive semi-definite and $s(0)=1$. Consider the rational mapping 
\begin{align*}
    \mu&:Y\times Y\dashrightarrow Y, \\
    \mu(x,y)&:=R\left(\frac{s(y)}{s(x)+s(y)}x+\frac{s(x)}{s(x)+s(y)}y\right).
\end{align*}

First of all we claim that the mapping extends to a regular one on a neighbourhood of $Y\times\{0\}$. Near $Y\times\{0\}$ it can be written as
\begin{equation}\label{rep_of_mu}
    \mu(x,y)=G\left(x,(y-x)(1+||x||^2)^{-N}\frac{q(x)}{s(x)+s(y)}\right).
\end{equation}

Keeping in mind that $s(0)=1$ and $s\geq 0$ we see that for $(x,y)\in Y\times\{0\}$ the following inequality holds:
\begin{multline*}
\left\Vert(y-x)(1+||x||^2)^{-N}\frac{q(x)}{s(x)+s(y)}\right\Vert=||x||(1+||x||^2)^{-N}\left|\frac{q(x)}{s(x)+1}\right|\leq\\
\leq ||x||(1+||x||^2)^{-N}|q(x)|.
\end{multline*}
This shows that the point $\left(x,(y-x)(1+||x||^2)^{-N}\frac{q(x)}{s(x)+s(y)}\right)$ belongs to $U$, so \eqref{rep_of_mu} gives a representation of $\mu$ as a regular mapping on a neighbourhood of $Y\times \{0\}$. 

Notice also that the function 
\begin{gather*}
\mu_0:Y\rightarrow Y, \\
\mu_0(x):=\mu(x,0)=G\left(x,-x(1+||x||^2)^{-N}\frac{1}{s(x)+1}\right)
\end{gather*}
is homotopic to the identity through the following homotopy relative to $\{0\}$:
\begin{align*}
    H&:Y\times [0,1]\rightarrow Y, \\
    H(x,t)&:=G\left(x,-tx(1+||x||^2)^{-N}\frac{1}{s(x)+1}\right).
\end{align*}

One can now analogously show that $\mu$ extends to a regular mapping on a neighbourhood of $\{0\}\times Y$ and that the mapping $\mu(0,y)$ considered as a function of $y$ is homotopic to the identity relative to $\{0\}$ as well. Since the property of being regular is local \cite[Proposition 3.2.3]{bochnakRealAlgebraicGeometry1998}, it follows that $\mu$ can be extended to a regular mapping defined on a neighbourhood $A$ of $Y\times \{0\}\cup \{0\}\times Y$. The extension satisfies the desired conditions.
\end{proof}
\end{thm}

\section{Proofs}

Theorems \ref{thm_groups} and \ref{thm_products} now follow easily:

\begin{proof}[Proof of Theorem \ref{thm_groups}]
Given a regular mapping $f:(\s^n,e)\rightarrow (Y,y_0)$, the class $-[f]\in\pi_n(Y,y_0)$ is represented by $f\circ \sym$, where $\sym:\s^n\rightarrow \s^n$ is any symmetry through a vector hyperplane passing through $e$. Hence $\pi_n^\alg(Y,y_0)$ is closed under taking inverse elements. 

Let now $\alpha_1,\alpha_2\in \pi_n^\alg(Y,y_0)$ be two algebraic homotopy classes. Choose $f_1,f_2$ to be their representatives such that $f_1$ is constantly equal to $y_0$ on the hemisphere $\{x_{n+1}\leq 0\}$, while $f_2$ is constantly equal to $y_0$ on the hemisphere $\{x_{n+1}\geq 0\}$. Then $\mu(f_1, f_2)$ is well defined and by definition represents the homotopy class $\alpha_1+\alpha_2$. Applying Theorem \ref{approx_thm}, we can find regular approximations $\widetilde{f_1},\widetilde{f_2}$ of $f_1,f_2$ respectively satisfying $\widetilde{f_1}(e)=\widetilde{f_2}(e)=y_0$. If the approximations are sufficiently close then $\mu(\widetilde{f_1},\widetilde{f_2})$ is well defined and close to $\mu(f_1,f_2)$, hence it represents the same homotopy class.
\end{proof}

\begin{proof}[Proof of Theorem \ref{thm_products}]
By the definition, given two continuous mappings $f:(\s^n,e)\rightarrow (Y,y_0),g:(\s^m,e)\rightarrow (Y,y_0)$ the Whitehead product of their classes is represented by the following composition
\begin{center}
    \begin{tikzcd}
h:\s^{n+m-1} \arrow[r, "\phi"] & \s^n\vee\s^m \arrow[r, "f\vee g"] & Y
\end{tikzcd}
\end{center}
where $\phi$ is the attaching map of the $n+m$-cell in the standard CW-structure of the torus $\s^n\times \s^m$. Let $\text{const}_n:\s^n\rightarrow Y,\;\text{const}_m:\s^m\rightarrow Y$ denote the mappings constantly equal to $y_0$, and define:
\begin{align*}
    \psi:=(f\vee \text{const}_m)\circ \phi, \\
    \xi:=(\text{const}_n\vee g)\circ \phi.
\end{align*}

It is then clear that $\psi$ and $\xi$ are both homotopic to the constant mapping, as they represent the homotopy classes of the products $[f,\text{const}_m]$ and $[\text{const}_n, g]$ respectively. Moreover it follows from the construction that for every $x\in \s^n$ we have
\begin{equation*}
\text{either }
\begin{cases}
\psi(x)=h(x),\\
\xi(x)=y_0,
\end{cases}
\text{or }
\begin{cases}
\psi(x)=y_0,\\
\xi(x)=h(x).
\end{cases}
\end{equation*}
This shows that $h$ is homotopic to $\mu(\psi,\xi)$ relative to $\{e\}$.

Applying Theorem \ref{approx_thm} we can find regular mappings 
\begin{align*}
    \tilde{\psi}:\s^{n+m-1}\rightarrow Y,\\
    \tilde{\xi}:\s^{n+m-1}\rightarrow Y,
\end{align*}
which satisfy $\psi(e)=\xi(e)=y_0$ and approximate $\psi$ (respectively $\xi$) closely. If the approximation is sufficiently close then $\tilde{h}:=\mu(\tilde{\psi},\tilde{\xi})$ is a close regular approximation of $\mu(\psi,\xi)$. In particular it represents the same homotopy class.
\end{proof}
\begin{rem} 
A close analysis of the proof shows, that if instead of $\s^{n+m-1}$ one considers a different real affine variety $\widetilde{\s}^{n+m-1}$ homeomorphic to a sphere with a distinguished point $\tilde{e}$, then it still holds that $[\alpha,\beta]$ is represented by a regular mapping $f:(\widetilde{\s}^{n+m-1},\tilde{e})\rightarrow (Y,y_0)$. This shows that certain homotopy classes of $(Y,y_0)$ are algebraic irrespective of what algebraic structure we equip the sphere with. We plan to study this interesting phenomenon closer in the future in another paper.
\end{rem}

Lastly, we prove Theorem \ref{thm_curves}. The proof is independent of the rest of the theory presented here in the paper, and it relies only on an old result of Bochnak and Kucharz.
\begin{proof}[Proof of Theorem \ref{thm_curves}]
Let $Y$ be a nonsingular retract rational variety. Thanks to the resolution of singularities, without loss of generality we may replacy $Y$ by its nonsingular compactification, so we assume that $Y$ is compact.

Let $U$ be an open subset of $Y$, let $V$ be an open subset of $\R^n$ for some $n$, so that there exist the two regular mappings $i:U\rightarrow V$ and $r:V\rightarrow U$ such that $r\circ i=\id_U$. Again thanks to the resolution of singularities, there exists a nonsingular compactification $Z$ of $V$, such that $r$ extends to a regular mapping $R:Z\rightarrow Y$. 

Let $X$ be a nonsingular compact curve and let $f:X\rightarrow Y$ be continuous. After a small perturbation we can assume that the image of $f$ is not contained in $Y\backslash U$. Also, thanks to \cite[Chapter VII, Theorem 2.25]{guaraldoTopicsRealAnalytic1986}, $f$ can be uniformly approximated by analytic mappings, so we can assume that it is analytic.

The composition $i\circ f$ is now a meromorphic mapping from $X$ into the compact variety $Z$. As $X$ is nonsingular and one-dimensional it follows that it admits an extension to an analytic mapping $g:X\rightarrow Z$, which then necessarily satisfies $R\circ g=f$. As the variety $Z$ is rational it follows from \cite{bochnakWeierstrassApproximationTheorem1999} that $g$ can be approximated by a regular mapping $\tilde{g}:X\rightarrow Z$. Then, $R\circ \tilde{g}$ is a regular approximation of $f$.
\end{proof}

\section{Counterexamples}
In this section we provide two counterexamples: the first one showing that Theorem \ref{thm_groups} fails without the assumption about $Y$ being retract rational, and the second one showing that the conclusion of Theorem \ref{thm_groups} cannot be strengthened to an equality $\pi_n^\alg(Y,y_0)=\pi_n(Y,y_0)$.
\begin{ex}\label{ex_1}
Let us start with a nonsingular curve $C\subset \R^2$ diffeomorphic to $\R$, with the property that that every regular mapping $f:\s^1\rightarrow C$ is constant. For example, one can take the curve explicitly defined by the equation
\begin{equation*}
    y^2=x(x^2+1).
\end{equation*}
Its Zariski closure inside $\mathbb C\mathbb P^2$ is a nonsingular complex projective curve of genus one, so it follows that every rational mapping from a rational curve into $C$ is constant.

Define $Z:=C\times C\subset \R^4$ and let $\sigma:Y\rightarrow Z$ be the composition of two blowups: the first one at a point $z_0\in Z$ and the second one at a point of the exceptional divisor of the first. We have that the exceptional locus $\sigma^{-1}(z_0)$ of $\sigma$ consists of two irreducible components $D_1$ and $D_2$ intersecting at a common point $y_0\in Y$, both isomorphic to the projective line $\R P^1$. 

Since $Z$ is nonsingular, it follows that $Y$ is nonsingular as well. We claim, however, that $\pi_1^\alg(Y,y_0)$ is not a subgroup of $\pi_1(Y,y_0)$. To verify that, let $f:(\s^1,e)\rightarrow (Y,y_0)$ be a regular mapping. By construction of $Z$ the composition $\sigma\circ f$ is constant, so the image of $f$ is contained in $\sigma^{-1}(z_0)=D_1\cup D_2$. Since $\s^1$ is irreducible, we must either have that $f(\s^1)\subset D_1$ or $f(\s^1)\subset D_2$. It follows that
\begin{equation*}
    \pi_1^\alg(Y,y_0)=i^1_\ast(\pi_1^\alg(D_1,y_0))\cup i^2_\ast(\pi_1^\alg(D_2,y_0)),
\end{equation*}
where $i^1:(D_1,y_0)\hookrightarrow (Y,y_0)$ and $i^2:(D_2,y_0)\hookrightarrow (Y,y_0)$ are the inclusions. According to Corollary \ref{cor:first_group} we have $\pi_1^\alg(D_1,y_0)=\pi_1(D_1,y_0)$ and $\pi_1^\alg(D_2,y_0)=\pi_1(D_2,y_0)$, so
\begin{equation*}
    \pi_1^\alg(Y,y_0)=i^1_\ast(\pi_1(D_1,y_0))\cup i^2_\ast(\pi_1(D_2,y_0)).
\end{equation*}

Note now, that since $Z$ is homeomorphic to $\R^2$, it retracts onto an arbitrarily small neighbourhood of $z_0$. Since $\sigma$ is proper and it restricts to a homeomorphism between $Y\backslash \sigma^{-1}(z_0)$ and $Z\backslash\{z_0\}$ it follows that $Y$ retracts onto an arbitrarily small neighbourhood of $\sigma^{-1}(z_0)$. From there, using the fact that $\sigma^{-1}(z_0)$ is an absolute neighbourhood retract, it can be retracted further onto $\sigma^{-1}(z_0)$. It follows that the homomorphism
\begin{equation*}
	i_\ast:\pi_1(D_1\cup D_2,y_0)\rightarrow \pi_1(Y,y_0)
\end{equation*}
has a left inverse, so it is injective. 

Since $D_1\cup D_2$ is homeomorphic to the wedge of two circles, the group $\pi_1(Y,y_0)$ is isomorphic to the free product $\mathbb Z\ast \mathbb Z$, where the two factors correspond to $\pi_1(D_1,y_0)$ and $\pi_1(D_2,y_0)$. It follows that $\pi_1(D_1,y_0)\cup \pi_2(D_1,y_0)$ is not a subgroup of $\pi_1(D_1\cup D_2,y_0)$, so after applying the injective homomorphism $i_\ast$ we get that its image is not a subgroup of $\pi_1(Y,y_0)$.
\end{ex}

\begin{ex}\label{ex_2}
This time we start with an irreducible nonsingular curve $C\subset \s^2$ with two connected components diffeomorphic to $\s^1$. For example, we can take the curve defined by the equation
\begin{equation*}
    \{(x,y,z)\in\s^2:z=y^2-6/5\}.
\end{equation*}
To verify that $C$ is indeed irreducible and nonsingular, note that the projection onto the first two coordinates gives an isomorphism between $C$ and 
\begin{equation*}
    C':=\{(x,y)\in \R^2:x^2+y^2+(y^2-6/5)^2=1\}.
\end{equation*}
The fact that $C'$ is nonsingular is easily checked using the Jacobian criterion, while the fact that it is irreducible follows from \cite[Theorem 4.5.1]{bochnakRealAlgebraicGeometry1998} after noticing that the polynomial $P=x^2+y^2+(y^2-6/5)^2-1$ is irreducible over $\R$.

Now, embed $\s^2$ into $\s^3$ to obtain an embedding of $C$ into $\s^3$. Let $Y$ be the variety obtained by blowing up $\s^3$ along $C$, and let $y_0\in Y$ be any point. Clearly, $Y$ is nonsingular and retract rational. We claim, however, that $\pi_2^\alg(Y,y_0)\neq \pi_2(Y,y_0)$.

To prove that, let us first study the topology of $Y$. Using Van Kampen's theorem one can show that the fundamental group of $Y$ is isomorphic to the free group with two generators. We omit the details concerning this calculation. We have that $H_2(F_2)=0$, where here $H_2$ denotes group homology (for a proof note that the wedge of two circles is a model of the Eilenberg-Maclane space $K(F_2,1)$). It hence follows from a theorem of Hopf on the cokernel of the second Hurewicz homomorphism (\cite[p. 257, Assertion a]{hopfFundamentalgruppeUndZweite1941}), that the homomorphism
\begin{equation*}
    h:\pi_2(Y,y_0)\rightarrow H_2(Y;\mathbb Z)
\end{equation*}
is surjective. 

Moreover, Hurewicz theorem and the isomorphism $\pi_1(Y,y_0)\cong F_2$ imply that $H_1(Y;\mathbb Z)$ is isomorphic to the free abelian group on two generators. Hence, from the universal coefficient theorem we conclude that the homomorphism induced by taking reduction modulo 2
\begin{equation*}
    H_2(Y;\mathbb Z)\rightarrow H_2(Y;\mathbb Z/2)
\end{equation*}
is surjective as well.

Now assume by contradiction that $\pi_2^\alg(Y,y_0)= \pi_2(Y,y_0)$. In the remaining part of the proof we will make use of the algebraic homology and cohomology groups of nonsingular compact affine varieties, for their definition see \cite[p. 273]{bochnakRealAlgebraicGeometry1998}. From functoriality of algebraic homology, we have that the composition of the Hurewicz homomorphism and reduction modulo 2
\begin{equation*}
    \pi_2(Y,y_0)\rightarrow H_2(Y;\mathbb Z)\rightarrow H_2(Y;\mathbb Z/2)
\end{equation*}
maps $\pi_2^\alg(Y,y_0)$ into $H_2^\alg(Y;\mathbb Z/2)$. As we have shown in the previous paragraph, the two homomorphisms are surjective, so the assumption $\pi_2^\alg(Y,y_0)= \pi_2(Y,y_0)$ implies that $H^\alg_2(Y;\mathbb Z/2)=H_2(Y;\mathbb Z/2)$, and hence $H^1_\alg(Y;\mathbb Z/2)=H^1(Y;\mathbb Z/2)$. Let us then proceed to show that this is not true.

Let $X$ denote the exceptional divisor of the blowup $Y\rightarrow \s^3$. Since the normal bundle of $C$ in $\s^3$ is trivial, we have that $X$ is isomorphic to $C\times \R P^1$. Let $C_1,C_2$ denote the two connected components of $C$. Denote by $\alpha \in H^1_\alg(Y;\mathbb Z/2)$ the cohomology class Poincar\'{e} dual to $[C\times \R P^1]_{\mathbb Z/2}$, the fundamental class of the submanifold $C\times \R P^1$ of $Y$.

Since $X$ is homeomorphic to the two-dimensional torus, we have that the first homology group $H_1(X;\mathbb Z/2)$ of $X$ is the free abelian group generated by the fundamental classes of the four submanifolds $\{c_1\}\times \R P^1,C_1\times \{a\},\{c_2\}\times \R P^1,C_2\times \{a\}$ of $X$, where $a\in \R P^1,c_1\in C_1,c_2\in C_2$ are any points. Perturbing the manifold $C_1\times \R P^1$ inside $Y$ a little, one verifies that the modulo 2 intersection number of $C_1\times \R P^1$ with $\{c_1\}\times \R P^1$ inside $Y$ is one, while the modulo 2 intersection number of $C_1\times \R P^1$ with any of the varieties $C_1\times \{a\},C_2\times \{a\},\{c_2\}\times \R P^1$ inside $Y$ is zero. It follows, that the pullback of $\alpha$ under the inclusion $X\hookrightarrow Y$ is equal to the class Poincar\'{e} dual to $[C_1\times \{a\}]_{\mathbb Z/2}$, the fundamental class of the submanifold $C_1\times \{a\}$ of $X$. From functoriality of algebraic cohomology, we find that $[C_1\times \{a\}]_{\mathbb Z/2}\in H_1^\alg(X;\mathbb Z/2)$. Projecting onto $C$, we get $[C_1]_{\mathbb Z/2}\in H_1^\alg(C;\mathbb Z/2)$. This is a contradiction, as irreducibility of $C$ implies that $H_1^\alg(C;\mathbb Z/2)$ is the subgroup of $H_1^\alg(X;\mathbb Z/2)$ generated by the fundamental class $[C]_{\mathbb Z/2}=[C_1]_{\mathbb Z/2}+[C_2]_{\mathbb Z/2}$.
\end{ex}
\subsection*{Acknowledgements}
The research was funded by the program Excellence Initiative – Research University at the Jagiellonian University in Krakow. The author thanks the referees for their careful reading of the manuscript and for many helpful suggestions.
\printbibliography

\end{document}